\documentclass[letterpaper,11pt,twoside,keywordsasfootnote,addressatend,noinfoline]{article}

\usepackage{fullpage}
\usepackage[english]{babel}
\usepackage{amssymb}
\usepackage{amsmath}
\usepackage{theorem}
\usepackage{epsfig}
\usepackage{subfigure}
\usepackage{mathrsfs}
\usepackage{imsart}

\theorembodyfont{\slshape}
\newtheorem{theorem}{Theorem}

\newtheorem{lemma}[theorem]{Lemma}

\newenvironment{proof}{\noindent{\scshape Proof.}}{\hspace*{2mm}~$\square$}
\newenvironment{demo}[1]{\noindent{\textbf{Proof of #1}}}{\hspace*{2mm}~$\square$}

\newcommand{\N}{\mathbb{N}}
\newcommand{\Z}{\mathbb{Z}}
\newcommand{\R}{\mathbb{R}}
\newcommand{\K}{\mathbb{K}}
\newcommand{\ind}{\mathbf{1}}
\newcommand{\ep}{\epsilon}
\newcommand{\ceil}[1]{\lceil{#1} \rceil}

\DeclareMathOperator{\card}{card}
\DeclareMathOperator{\poisson}{Poisson \,}
\DeclareMathOperator{\exponential}{Exponential \,}

\begin{document}

\begin{frontmatter}
\title     {Stochastic spatial model for the division \\ of labor in social insects}
\runtitle  {Stochastic spatial model for the division of labor}
\author    {Alesandro Arcuri and Nicolas Lanchier\thanks{Research supported in part by NSA Grant MPS-14-040958.}}
\runauthor {A. Arcuri and N. Lanchier}
\address   {School of Mathematical \\ and Statistical Sciences \\ Arizona State University \\ Tempe, AZ 85287, USA.}
\address   {School of Mathematical \\ and Statistical Sciences \\ Arizona State University \\ Tempe, AZ 85287, USA.}

\maketitle

\begin{abstract} \ \
 Motivated by the study of social insects, we introduce a stochastic model based on interacting particle systems in order to understand the effect of communication on the division of labor. Members of the colony are located on the vertex set of a graph representing a communication network. They are characterized by one of two possible tasks, which they update at a rate equal to the cost of the task they are performing by either defecting by switching to the other task or cooperating by anti-imitating a random neighbor in order to balance the amount of energy spent in each task. We prove that, at least when the probability of defection is small, the division of labor is poor when there is no communication, better when the communication network consists of a complete graph, but optimal on bipartite graphs with bipartite sets of equal size, even when both tasks have very different costs. This shows a non-monotonic relationship between the number of connections in the communication network and how well individuals organize themselves to accomplish both tasks equally.
\end{abstract}



\end{frontmatter}


\section{Introduction}
\label{sec:intro}

\indent This work is primarily motivated by the study of social insects such as ants, honey bees, wasps and termites~\cite{wilson_1971} but applies more generally to any population whose individuals have the ability to exchange
 information in order to get organized socially.
 More specifically, we are interested in the following two fundamental components of social insects colonies. \vspace*{5pt}

\noindent {\bf Communication system.}
 Social insects are characterized by their well-developed communication system.
 For instance, ants communicate with each other using some pheromones that they most often leave on the soil surface to mark a trail leading from the colony to a food source.
 But pheromones are also used by ants to let other colony members know what task group they belong to.
 Ants can also communicate by direct contact, using for instance their antennae. \vspace*{5pt}

\noindent {\bf Division of labor.}
 The well-developed communication system of social insects is central to complex social behaviors, including division of labor~\cite{beshers_fewell_2001, wilson_1980}, i.e., cooperative work.
 Returning to the example of ants, except for the queens and reproductive males, the other individuals work together to create a favorable environment for the colony and the brood.
 The repertoire of tasks includes nest maintenance, foraging, brood care and nest defense, which have different costs.
 While some ants may specialize on a given task, which depends on their age class and morphology, most of the workers are totipotent~\cite{beshers_fewell_2001, wilson_1971}, meaning that they are able to perform all tasks and can
 therefore switch from one task to another in response to the need of the colony. \vspace*{5pt}

\noindent For a review on the various models of division of labor, we refer to~\cite{beshers_fewell_2001}.
 There, the authors distinguish six classes of models:
 response threshold, integrated threshold-information transfer, self-reinforcement, foraging for work, social inhibition, and network task allocation models.
 These six different classes of models are built on six different assumptions about the causes of division of labor in social insects.
 The stochastic model introduced in this paper belongs to the last class of models: network task allocation models~\cite{gordon_1992, pacala_1974}.
 Our objective is not to speculate on the communication system or on the division of labor in social insects, which is the work of experimental ecologists, but instead to understand how the communication system may affect the division of labor.
 More precisely, we think of the colony as being spread out on the vertex set of a graph representing a communication network in the sense that two individuals are given the opportunity to communicate if and only if the
 corresponding vertices are connected by an edge, and study the effects of the topology of the communication network on the division of labor.

\indent To design such a network task allocation model, it is natural to use the mathematical framework of interacting particle systems~\cite{durrett_1995, liggett_1985, liggett_1999}, which describe the interactions among particles
 or individuals located on a graph.
 The main objective of research in this field is to deduce the macroscopic behavior at the colony level that emerges from the microscopic rules at the individual level, which usually strongly depends on the topology of the network of interactions.

\indent One of the simplest and most popular example of interacting particle systems is the voter model, independently introduced in~\cite{clifford_sudbury_1973, holley_liggett_1975}.
 In this model, individuals are characterized by their opinion that they update at rate one by imitating one of their neighbors chosen uniformly at random.
 To obtain a model for the dynamics of tasks, we think of each individual as being characterized by the task they are performing rather than their opinion, and use instead a variant of the much less popular
 anti-voter model~\cite{matloff_1977}, also called the dissonant voting model, where individuals anti-imitate a neighbor chosen uniformly at random.
 The general idea is that individuals have no information about the overall division of labor beyond knowing which tasks their neighbors on the communication network are performing so, in order to balance the amount of energy
 spent in each task, they will try to perform a different task from their neighbors.
 Though our model has some similarities with previous network task allocation models, this paper relies on analytical results rather than just simulations and is therefore more mathematical in spirit than previous works on this topic.


\section{Model description}
\label{sec:model}

\indent To construct our model for the dynamics of tasks, we first let~$G = (V, E)$ be a graph representing a communication network.
 The colony is spread out on the vertex set of this graph, with exactly one individual at each of the vertices, and two individuals are given the opportunity to communicate if and only if there is an edge connecting the
 corresponding two vertices.
 Individuals are characterized by the task they are performing, which they update at random times based on the information they get from their neighbors on the graph.
 For simplicity, we assume that the communication network is static, that there are only two tasks to be performed and that each individual is always performing exactly one task.
 As previously mentioned, it is natural to employ the framework of interacting particle systems in which we keep track of the task performed by each individual using a continuous-time Markov chain
 whose state at time~$t$ is a configuration
 $$ \xi_t : V \to \{1, 2 \} = \hbox{set of tasks to be performed} $$
 with~$\xi_t (x)$ denoting the task performed at time~$t$ by the individual at vertex~$x$.
 The evolution rules depend on a couple of parameters.
 First, we assume that performing task~$i$ has a cost~$c_i > 0$, which is included in the dynamics by interpreting the cost of a task as the rate at which an individual performing this task attempts to switch to the other task.
 We also assume that both tasks are equally important for the survival of the colony so, in the best case scenario, each time an individual communicates with a neighbor, it will cooperate by anti-imitating this neighbor in
 order to balance the amount of energy spent in each task.
 We consider an additional parameter~$\ep$ that we interpret as probability of defection and is included in the dynamics by assuming that, each time an individual wants to switch task,
\begin{list}{\labelitemi}{\leftmargin=1.75em}
 \item with probability~$\ep$, the individual defects by switching to the other task without communicating with any of its neighbors whereas \vspace*{3pt}
 \item with probability~$1 - \ep$, the individual cooperates by communicating with a random neighbor and anti-imitating this neighbor. \vspace*{-1pt}
\end{list}
 Note that individuals with no neighbor have no other choice than always defecting due to their lack of knowledge.
 To describe the evolution rules formally, we let
 $$ N_x := \{y \in V : (x, y) \in E \} \quad \hbox{for all} \quad x \in V $$
 be the interaction neighborhood of vertex~$x$ and
 $$ f_i (x, \xi) := \card \,\{y \in N_x : \xi (y) = i \} / \card (N_x) \quad \hbox{for} \quad i = 1, 2, $$
 be the fraction of neighbors of vertex~$x$ that are performing task~$i$ when the system is in configuration~$\xi$, which we assume to be zero when~$x$ has no neighbor.
 Then, the transition rates at vertex~$x$ described above verbally can be written into equations as
\begin{equation}
\label{eq:rate-labor}
  \begin{array}{rcl}
     1 \ \to \ 2 & \hbox{at rate} & c_1 \,(\ep + (1 - \ep)(1 - f_2 (x, \xi))) \vspace*{3pt} \\
     2 \ \to \ 1 & \hbox{at rate} & c_2 \,(\ep + (1 - \ep)(1 - f_1 (x, \xi))). \end{array}
\end{equation}
 By our convention (the fraction of neighbors is zero when there is no neighbor), when a vertex has no neighbor, it switches from task~$i$ to the other task at rate~$c_i$ in agreement with
 our verbal description of the model.
 Note also that the anti-voter model~\cite{matloff_1977} is simply obtained by setting~$c_1 = c_2$ and~$\ep = 0$, though this special case is not of particular interest in our biological context. \vspace*{5pt}


\noindent {\bf Division of labor.}
 The main objective is to understand how the topology of the communication network affects the division of labor, which we model using the random variable
\begin{equation}
\label{eq:division-labor}
  \phi (s) := \frac{1}{s \,\card (V)} \,\int_0^s X_t \,dt \quad \hbox{where} \quad X_t := \sum_{x \in V} \,\ind \{\xi_t (x) = 1 \}.
\end{equation}
 The process~$X_t$ keeps track of the number of individuals performing task~1 at time~$t$.
 We point out that this is not a Markov process in general because the rate at which the number of individuals performing a given task varies depends on the specific location of these individuals on the communication network.
 The random variable~$\phi (s)$ represents the fraction of time up to time~$s$ and averaged across all the colony task~1 has been performed.
 Lemma~\ref{lem:general} below will show that, at least when the defection probability is positive, this random variable converges almost surely to a limit that does not depend on the initial configuration of the system.
 Under our assumption that both tasks are equally important for the survival of the colony, the division of labor is optimal when this limit is one half and poor when the limit is close to either zero or one.


\section{Main results}
\label{sec:results}

\indent To fix the ideas, we assume from now on without loss of generality that the first task is the less costly therefore we have~$c_1 < c_2$.
 To start with a reference value, assume that the communication network is completely disconnected, i.e., there is no edge, meaning no communication.
 In this case, an individual performing task~$i$ switches to the other task at rate~$c_i$ so, by independence,
 $$ \begin{array}{l} \lim_{s \to \infty} \,\phi (s) = \lim_{t \to \infty} \,P \,(\xi_t (x) = 1) = \bar v_1 := c_2 \,(c_1 + c_2)^{-1} > 1/2 \end{array} $$
 almost surely for all~$x \in V$.
 Note that the division of labor converges almost surely to the same limit~$\bar v_1$ for all communication networks when the defection probability is equal to one, since in this case the individuals never communicate with their neighbors.

\indent Looking now at more complex communication networks, Figure~\ref{fig:frame} shows the limit of the division of labor as a function of the defection probability obtained from numerical simulations of the process
 on each of the first three graphs of Figure~\ref{fig:graphs} but with one thousand instead of eight vertices like in the picture.
 As expected, the simulation results suggest that the division of labor gets improved, moving closer to one half, as the defection probability decreases and converges almost surely to the reference value~$\bar v_1$ as
 the defection probability increases to one.
 In addition, at least for small defection probabilities, the division of labor is better for the process on large complete graphs than for the process on completely disconnected graphs.
 However, the simulation results also suggest something less obvious:
 the division of labor is much better on the one-dimensional torus where individuals can only communicate with their two nearest neighbors than on the complete graph where individuals can communicate with every other individual.
 This reveals a non-monotonic relationship between the number of connections and how well individuals organize themselves to accomplish both tasks equally, which we now prove analytically.

\begin{figure}[t]
 \centering
 \scalebox{0.50}{\input{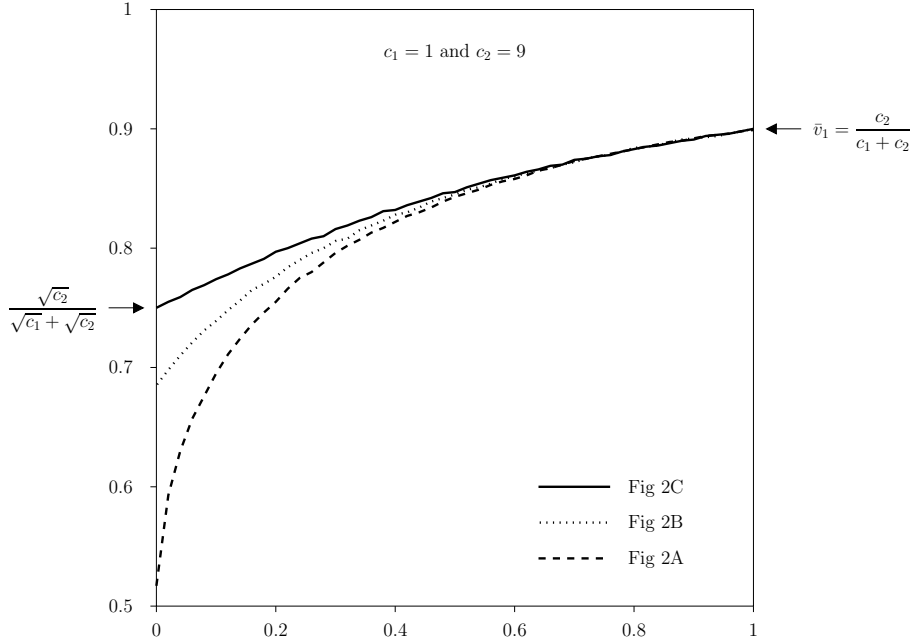}}
 \caption{\upshape Division of labor averaged over~$10^6$ updates of the system as a function of the probability of defection for the first three graphs of Figure~\ref{fig:graphs} with one thousand instead of eight vertices.
   For each graph, the curve is obtained from a single realization of the process starting with all individuals performing task~1.}
\label{fig:frame}
\end{figure}

\begin{figure}[t]
\centering
\includegraphics[width=0.70\textwidth]{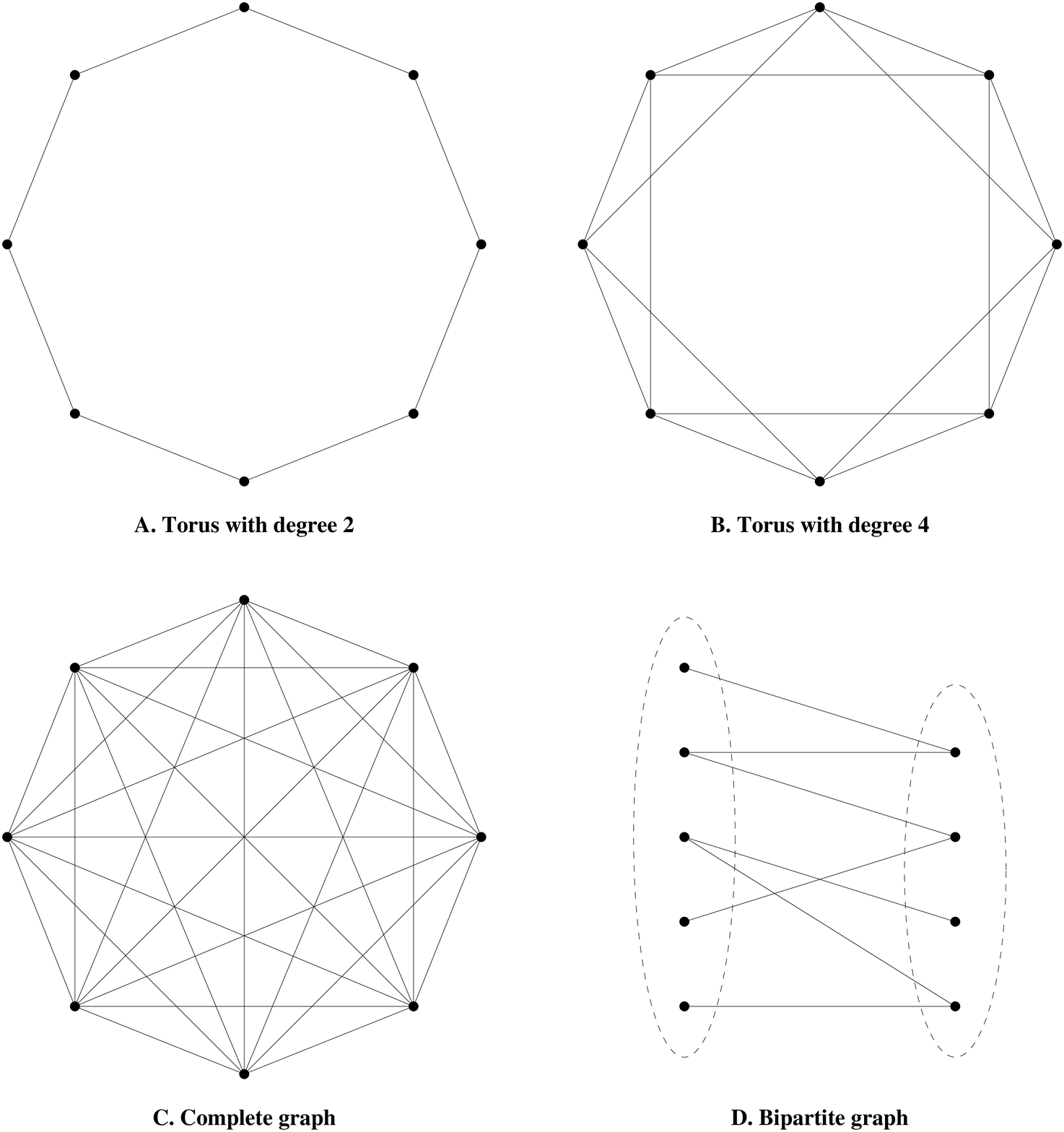}
\caption{\upshape{Examples of communication networks.
  The simulation curves in Figure~\ref{fig:frame} are for the first three networks: the tori with degree~2 and degree~4 and the complete graph.
  The process on complete graphs and on bipartite graphs, which includes the torus with degree~2 and an even number of vertices, are also studied analytically.}}
\label{fig:graphs}
\end{figure}

\indent To begin with, we give the exact limit of the division of labor for the process on complete graphs for all values of the defection probability.
 In particular, we obtain an explicit expression for the equation of the curve in solid line in Figure~\ref{fig:frame}.
 Then, we give lower and upper bounds for the limit of the division of labor when the defection probability is small and for the process on bipartite graphs, which will also give a proof that the dashed curve
 in Figure~\ref{fig:frame} indeed starts at one half using that the ring with an even number of vertices is an example of a bipartite graph.
 Finally, we will prove one more result for the process on the one-dimensional integer lattice to extend to infinite graphs the results obtained for finite bipartite graphs. \vspace*{5pt}


\noindent {\bf Complete graphs.}
 To begin with, we examine the process on a complete graph.
 We recall that the complete graph with~$N$ vertices, denoted by~$\K_N$, is such that
 $$ (x, y) \in E \quad \hbox{if and only if} \quad x, y \in V \ \hbox{with} \ x \neq y $$
 indicating that all the individuals can communicate with each other.
 From a modeling perspective, this is realistic for well-mixing colonies.
 In this case, we have the following result.
\begin{theorem} --
\label{th:complete}
 Let~$G = \K_N$ and~$\ep > 0$.
 Let~$B := (1 - \ep)(1 - 1/N)^{-1}$.
\begin{itemize}
\item Regardless of the initial state,
\begin{equation}
\label{eq:complete}
  \begin{array}{l} \lim_{s \to \infty} \,\phi (s) = \bar u_1 (B) := \displaystyle \frac{1}{2} - \frac{1}{2B} \,\bigg(\frac{c_1 + c_2}{c_1 - c_2} + \sqrt{(B - 1)^2 + \frac{4 \,c_1 c_2}{(c_1 - c_2)^2}} \bigg). \end{array}
\end{equation}
\item The function~$B \mapsto \bar u_1 (B)$ is decreasing on~$(0, 2)$ and
  $$ \begin{array}{l} \lim_{B \to 0} \,\bar u_1 (B) = \bar v_1 = \displaystyle \frac{c_2}{c_1 + c_2} > \bar u_1 (1) = \frac{\sqrt{c_2}}{\sqrt{c_1} + \sqrt{c_2}} > \bar u_1 (2) = 1/2. \end{array} $$
\end{itemize}
\end{theorem}
 Noticing that~$B (\ep, N)$ is decreasing with respect to both the probability of defection and the colony size, the monotonicity of the function~$\bar u_1$ implies the following.
\begin{itemize}
 \item The smaller the probability of defection~$\ep$, the better the division of labor, i.e., the fraction of individuals performing any given task gets closer to one half.
  This conclusion is expected from the definition of our model: cooperation improves the division of labor. \vspace*{3pt}
 \item The smaller the colony size~$N$, the better the division of labor.
  The intuition behind this result is that, at each update of an individual, the knowledge this individual gets about the tasks performed by the rest of the colony is only through a single random individual therefore
  the amount of information obtained from each interaction about the overall division of labor gets smaller as the size of the colony increases.
\end{itemize}
 The three different values of the limiting fraction of individuals performing task~1 in the second part of the theorem are also interesting from a biological point of view.
\begin{itemize}
 \item The limit as~$B \to 0$, i.e., $\ep \to 1$, is not obvious because~\eqref{eq:complete} is not defined at~$B = 0$.
  However, the process is well defined when~$\ep = 1$.
  In this case, there is no cooperation so the process evolves as when the communication network is completely disconnected, which explains the value of the limit: $\bar v_1 = c_2 \,(c_1 + c_2)^{-1}$. \vspace*{3pt}
 \item The value at~$B = 1$ is the most interesting one biologically because it gives the limiting behavior when the colony is large and the probability of defection is small.
  In this case, cooperation acts positively by making the fraction of individuals performing any given task closer to one half compared to the scenario where~$\ep = 1$. \vspace*{3pt}
 \item The value at~$B = 2$ corresponds to~$\ep = 0$ and~$N = 2$ and can be easily guessed:
  there is perfect cooperation and only two individuals so the two states where one individual performs task~1 and the other task~2 are absorbing states, which gives the value~1/2.
\end{itemize} \vspace*{5pt}


\noindent {\bf Bipartite graphs.}
 Theorem~\ref{th:complete} shows that, at least for complete graphs, the division of labor is optimal, i.e., converges to one half, when there are only two vertices and the defection probability converges to zero.
 We now extend this result to finite, connected, bipartite graphs.
 Recall that a graph is bipartite when there exists a partition~$\{V_1, V_2 \}$ of the vertex set into so-called bipartite sets such that vertices in the same bipartite set cannot be neighbors, i.e.,
 $$ (x, y) \in E \quad \hbox{implies that} \quad (x, y) \in V_1 \times V_2 \quad \hbox{or} \quad (x, y) \in V_2 \times V_1. $$
 In other words, individuals are separated into two groups and only communicate with members of the opposite group (see Figure~\ref{fig:graphs}D).
 At least when the defection probability is small, the dynamics of the process tends to separate the two tasks in each of the two bipartite sets.
 More precisely, the division of labor is bounded from above and from below as follows.
\begin{theorem} --
\label{th:bipartite}
 Assume that~$G$ has~$N$ vertices, is connected and bipartite. Then,
\begin{itemize}
 \item for all~$\rho > 0$, there exists~$\ep_0 > 0$ such that
 $$ \begin{array}{l} (1 - \rho) \,\min (N_1/N, N_2/N) \leq \lim_{s \to \infty} \,\phi (s) \leq (1 + \rho) \,\max (N_1/N, N_2/N) \end{array} $$
 almost surely for all~$\ep < \ep_0$, where~$N_i = \card (V_i)$ for~$i = 1, 2$.
\end{itemize}
\end{theorem}
 Note that a complete graph is a bipartite graph if and only if it has only two vertices.
 In this case, the previous theorem implies that, for all~$\ep < \ep_0$,
\begin{equation}
\label{eq:bipartite-half}
  \begin{array}{l} (1/2)(1 - \rho) \leq \lim_{s \to \infty} \,\phi (s) \leq (1/2)(1 + \rho), \end{array}
\end{equation}
 showing that the division of labor converges to one half as~$\ep \to 0$.
 In particular, Theorem~\ref{th:bipartite} indeed extends the last limit in the statement of Theorem~\ref{th:complete}.
 Figure~\ref{fig:frame}A and more generally any torus with degree two and an even number of vertices are other examples of bipartite graphs where both bipartite sets have the same cardinality.
 In particular, the theorem implies that~\eqref{eq:bipartite-half} holds for all these graphs, which gives a proof of the simulation results of Figure~\ref{fig:frame} for the dashed curve in the neighborhood of~$\ep = 0$.
 Another example of communication network more relevant for ecologists is to assume that the colony is spread out on the two-dimensional grid with vertex set
 $$ V := \Z^2 \cap \{[0, L] \times [0, H] \} \quad \hbox{where} \quad L, H \in \N^* $$
 and where there is an edge between vertices at Euclidean distance one from each other.
 In this context, the individuals are mostly static and can only communicate with individuals which are close to them.
 This is an example of bipartite graph.
 The bipartite sets are given by
 $$ V_1 := \{x \in V : x_1 + x_2 \ \hbox{is odd} \} \quad \hbox{and} \quad V_2 := \{x \in V : x_1 + x_2 \ \hbox{is even} \}. $$
 These two sets differ in cardinality by at most one so
 $$ \begin{array}{l} (1 - \rho)(1/2 - 1/N) \leq \lim_{s \to \infty} \,\phi (s) \leq (1 + \rho)(1/2 + 1/N), \end{array} $$
 for all~$\ep > 0$ small according to Theorem~\ref{th:bipartite}, showing that, at least when~$\ep$ is small and the colony size is large, the division of labor again approaches one half.
 Since these bipartite graphs have more edges than completely disconnected graphs but less edges than complete graphs, this shows a non-monotonic relationship between the degree distribution of the communication network
 and how well individuals organize themselves to accomplish both tasks equally. \vspace*{5pt}


\noindent {\bf One-dimensional lattice.}
 The last communication network we consider is the one-dimensional integer lattice:
 individuals are arranged in a line and can only communicate with their two nearest neighbors on the left and on the right.
 Our main motivation is to show that more sophisticated techniques can be used to study the process on a simple infinite graph.
 In this case, the existence of the process follows from a general result due to Harris~\cite{harris_1972}.
 Because the graph is infinite, the random variables in~\eqref{eq:division-labor} are no longer well-defined therefore we study instead the division of labor in an increasing sequence of spatial intervals
 $$ \phi_N (s) := \frac{1}{s \,(2N + 1)} \,\int_0^s \sum_{|x| \leq N} \,\ind \{\xi_t (x) = 1 \} \,dt $$
 as time~$s$ and space~$N$ both go to infinity.
 Then, for the process on the infinite one-dimensional lattice, we have the following result that extends in part Theorem~\ref{th:bipartite}.
\begin{theorem} --
\label{th:lattice}
 Assume that~$G$ is the one-dimensional integer lattice and that the initial distribution of the process is a Bernoulli product measure. Then,
\begin{itemize}
 \item for all~$\rho > 0$, there exists~$\ep_0 > 0$ such that
 $$ \begin{array}{l} (1/2)(1 - \rho) \leq \lim_{s, N \to \infty} \,\phi (s, N) \leq (1/2)(1 + \rho) \quad \hbox{for all} \quad \ep < \ep_0. \end{array} $$
\end{itemize}
\end{theorem}
 The theorem states that, when the defection probability is small, the fraction of individuals performing task~1 in a large spatial interval approaches one half.
 Even though the one-dimensional lattice is a bipartite graph, the result does not simply follow from Theorem~\ref{th:bipartite} because the two bipartite sets are infinite.
 Instead, the proof relies on a coupling between the dynamics of tasks and a system of annihilating random walks together with large deviation estimates.


\section{Proof of Theorem~\ref{th:complete} (complete graphs)}
\label{sec:complete}

\indent This section is devoted to the proof of Theorem~\ref{th:complete} which assumes that the communication network is a complete graph.
 To prove that the division of labor converges almost surely to a limit that does not depend on the initial state, we start with a general lemma about finite graphs that will be used in this section and the next one.
 Now, in the special case of complete graphs, the process~$X_t$ that keeps track of the number of individuals performing task~1 turns out to be a continuous-time birth and death process.
 From this, we will deduce that the expected fraction of individuals performing task~1 satisfies a certain ordinary differential equation.
 The exact value of the almost sure limit of the division of labor is then obtained by showing that this limit is the unique fixed point in the biologically relevant interval~$(0, 1)$ of the differential equation.
\begin{lemma} --
\label{lem:general}
 Let~$G$ be finite and~$\ep > 0$.
 Then~$\phi (s)$ converges almost surely to a limit that does not depend on the initial configuration of the stochastic process.
\end{lemma}
\begin{proof}
 Let~$\xi$ be a configuration and~$x$ be a vertex, and assume that~$\xi (x) = i$.
 Since~$\ep > 0$, the rate at which vertex~$x$ switches task is bounded from below by
 $$ c_i \,(\ep + (1 - \ep)(1 - f_j (x, \xi))) \geq \ep \,c_i \geq \ep \,c_1 > 0 \quad \hbox{where} \quad \{i, j \} = \{1, 2 \}. $$
 This shows that the process is irreducible.
 Since in addition the communication network is finite, the state space of the process is finite as well.
 In particular, standard Markov chain theory implies that there is a unique stationary distribution~$\pi$ to which the process converges weakly starting from any initial configuration.
 Moreover, we have the almost sure convergence
\begin{equation}
\label{eq:general}
 \begin{array}{l} \lim_{s \to \infty} \,\phi (s) = E_{\pi} (\sum_{x \in V} \ind \{\xi (x) = 1 \} / N) \end{array}
\end{equation}
 where the integer~$N$ denotes the number of vertices.
\end{proof} \\ \\
 In the next lemma, we compute the limit~\eqref{eq:general} for complete graphs.
 Following the notation of the theorem, we let~$\bar u_1 (B)$ denote this limit where~$B = (1 - \ep)(1 - 1/N)^{-1}$.
\begin{lemma} --
\label{lem:complete-limit}
 For all~$\ep \in (0, 1]$ and~$N \geq 2$, we have
\begin{equation}
\label{eq:complete-limit-0}
  \bar u_1 (B) = \frac{1}{2} - \frac{1}{2B} \,\bigg(\frac{c_1 + c_2}{c_1 - c_2} + \sqrt{(B - 1)^2 + \frac{4 \,c_1 c_2}{(c_1 - c_2)^2}} \bigg).
\end{equation}
\end{lemma}
\begin{proof}
 The communication network being a complete graph, the transition rates~\eqref{eq:rate-labor} are invariant by permutation of the vertices therefore the spatial configuration of individuals is no longer relevant and the
 process~$X_t$ itself is Markov.
 More precisely, since the state can only increase or decrease by one unit at a positive rate, this process is a continuous-time birth and death process, thus characterized by its birth and death parameters
 $$ \begin{array}{rcl}
      \beta_j & := & \lim_{h \to 0} \,(1/h) \,P \,(X_{t + h} = j + 1 \,| \,X_t = j) \vspace*{4pt} \\
     \delta_j & := & \lim_{h \to 0} \,(1/h) \,P \,(X_{t + h} = j - 1 \,| \,X_t = j). \end{array} $$
 By standard properties of independent Poisson processes, the birth parameter~$\beta_j$ is obtained by multiplying the number of individuals performing task~2 by the common rate at which each of these individuals switches to task~1.
 Together with~\eqref{eq:rate-labor}, this gives
\begin{equation}
\label{eq:birth}
  \beta_j = c_2 \,(N - j) \bigg(\ep + (1 - \ep) \bigg(\frac{N - j - 1}{N - 1} \bigg) \bigg) \quad \hbox{for} \quad j = 0, 1, \ldots, N.
\end{equation}
 By symmetry, the death parameters are given by
\begin{equation}
\label{eq:death}
  \delta_j = c_1 \,j \,\bigg(\ep + (1 - \ep) \bigg(\frac{j - 1}{N - 1} \bigg) \bigg) \quad \hbox{for} \quad j = 0, 1, \ldots, N.
\end{equation}
 The next step is to use the transition rates~\eqref{eq:birth}--\eqref{eq:death} to derive a differential equation for the expected value of the fraction of individuals performing a given task at time~$t$, namely
 $$ u_1 (t) := E \,(X_t / N) \quad \hbox{and} \quad u_2 (t) := E \,(1 - X_t / N). $$
 Even though these functions depend on the initial state, their limit as time goes to infinity does not according to Lemma~\ref{lem:general}.
 To compute the limit, we first observe that
 $$ \begin{array}{rcl}
      u_1' & = & \lim_{h \to 0} \ (1/h)(u_1 (t + h) - u_1 (t)) \vspace*{4pt} \\
           & = & \lim_{h \to 0} \ (1/Nh) \,E \,(X_{t + h} - X_t \,| \,X_t = N u_1 (t)) \vspace*{4pt} \\
           & = &  (1/N) \,\lim_{h \to 0} \ (1/h) \,P \,(X_{t + h} = N u_1 (t) + 1 \,| \,X_t = N u_1 (t)) \vspace*{4pt} \\
           &   & \hspace*{25pt} - \ (1/N) \,\lim_{h \to 0} \ (1/h) \,P \,(X_{t + h} = N u_1 (t) - 1 \,| \,X_t = N u_1 (t)) \vspace*{4pt} \\
           & = &  (1/N) \,\beta_{Nu_1 (t)} - (1/N) \,\delta_{Nu_1 (t)}. \end{array} $$
 Using also~\eqref{eq:birth}--\eqref{eq:death}, we deduce that
 $$ \begin{array}{rcl}
      u_1' & = & \displaystyle c_2 \,u_2 \,\bigg(\ep + (1 - \ep) \bigg(1 - \frac{Nu_1}{N - 1} \bigg) \bigg) - c_1 \,u_1 \,\bigg(\ep + (1 - \ep) \bigg(1 - \frac{Nu_2}{N - 1} \bigg) \bigg) \vspace*{8pt} \\
           & = & \displaystyle c_2 \,u_2 - c_1 \,u_1 + (1 - \ep) \bigg(\frac{1}{1 - 1/N} \bigg)(c_1 - c_2) \,u_1 \,u_2 =: Q (u_1, u_2). \end{array} $$
 Recalling that~$B = (1 - \ep)(1 - 1/N)^{-1}$ and~$u_1 + u_2 = 1$, we get
 $$ \begin{array}{rcl}
      Q (u_1, u_2) & = & Q (u_1, 1 - u_1) \vspace*{3pt} \\
                   & = & c_2 \,(1 - u_1) - c_1 \,u_1 + B \,(c_1 - c_2)(1 - u_1) \,u_1 \vspace*{3pt} \\
                   & = & c_2 - (c_1 + c_2) \,u_1 + B \,(c_1 - c_2) \,u_1 - B \,(c_1 - c_2) \,u_1^2. \end{array} $$
 Since this is a polynomial with degree two in~$u_1$ and that
 $$ Q (0, 1) = c_2 > 0 \quad \hbox{and} \quad Q (1, 0) = - c_1 < 0, $$
 we obtain the existence of a unique fixed point in~$(0, 1)$, namely~$\bar u_1 (B)$, which is globally asymptotically stable.
 Basic algebra shows that the discriminant is
 $$ \begin{array}{rcl}
      \Delta & = & (B \,(c_1 - c_2) - (c_1 + c_2))^2 + 4B \,c_2 \,(c_1 - c_2) \vspace*{3pt} \\
             & = &  B^2 \,(c_1 - c_2)^2 + (c_1 + c_2)^2 - 2B \,(c_1 + c_2)(c_1 - c_2) + 4B \,c_2 \,(c_1 - c_2) \vspace*{3pt} \\
             & = &  B^2 \,(c_1 - c_2)^2 + (c_1 + c_2)^2 - 2B \,(c_1 - c_2)^2 \vspace*{3pt} \\
             & = &  B^2 \,(c_1 - c_2)^2 + (c_1 - c_2)^2 - 2B \,(c_1 - c_2)^2 + 4 \,c_1 c_2 \vspace*{3pt} \\
             & = & (B - 1)^2 \,(c_1 - c_2)^2 + 4 \,c_1 c_2 \end{array} $$
 from which it follows that~$u_1 (t) = E \,(X_t / N)$ converges to
 $$ \begin{array}{rcl}
    \bar u_1 (B) & = & \displaystyle \frac{B \,(c_1 - c_2) - (c_1 + c_2)}{2B \,(c_1 - c_2)} - \frac{1}{2B} \,\sqrt{\frac{\Delta}{(c_1 - c_2)^2}} \vspace*{8pt} \\
                 & = & \displaystyle \frac{1}{2} - \frac{1}{2B} \,\bigg(\frac{c_1 + c_2}{c_1 - c_2} + \sqrt{(B - 1)^2 + \frac{4 \,c_1 c_2}{(c_1 - c_2)^2}} \bigg). \end{array} $$
 This completes the proof.
\end{proof} \\ \\
 The first part of Theorem~\ref{th:complete} follows from~Lemmas~\ref{lem:general}--\ref{lem:complete-limit}.
 We now deal with the second part, i.e., we study the function~$\bar u_1 (B)$, which gives some insight into the role of the probability of defection and the size of the colony in the division of labor.
 The monotonicity of this function and its value at biologically relevant values of~$B$ are established in the next four lemmas.
\begin{lemma} --
\label{lem:decrease}
  The function~$B \mapsto \bar u_1 (B)$ is decreasing on~$(0, 2)$.
\end{lemma}
\begin{proof}
 To begin with, we observe that
\begin{equation}
\label{eq:decrease-1}
  \begin{array}{l}
   \displaystyle \sqrt{(B - 1)^2 + \frac{4 \,c_1 c_2}{(c_1 - c_2)^2}} \ \leq \ \sqrt{1 + \frac{4 \,c_1 c_2}{(c_1 - c_2)^2}} \vspace*{8pt} \\ \hspace*{50pt} = \
   \displaystyle \sqrt{\bigg(\frac{c_1 - c_2}{c_1 - c_2} \bigg)^2 + \frac{4 \,c_1 c_2}{(c_1 - c_2)^2}} \ = \ - \ \frac{c_1 + c_2}{c_1 - c_2}. \end{array}
\end{equation}
 Then, we distinguish two cases. \vspace*{5pt} \\
 Assume first that~$B < 1$.
 Inequality~\eqref{eq:decrease-1} implies that
 $$ \bar u_1 (B) = \frac{1}{2} - \frac{1}{2B} \bigg(\frac{c_1 + c_2}{c_1 - c_2} \bigg) - \frac{1}{2B} \,\sqrt{(B - 1)^2 + \frac{4 \,c_1 c_2}{(c_1 - c_2)^2}} \geq 1/2 $$
 showing that~$\bar u_1 (B)$ is the unique solution~$X \in [1/2, 1]$ of
 $$ P \,(B, X) = c_2 - (c_1 + c_2) \,X + B \,(c_1 - c_2) \,X - B \,(c_1 - c_2) \,X^2 = 0. $$
 Since in addition, for all~$B \in (0, 1)$ and~$X \in [1/2, 1]$,
 $$ \begin{array}{rcl}
    \partial_B P \,(B, X) & = & (c_1 - c_2) \,X - (c_1 - c_2) \,X^2 = (c_1 - c_2)(1 - X) \,X < 0 \vspace*{5pt} \\
    \partial_X P \,(B, X) & = & - \ (c_1 + c_2) + B \,(c_1 - c_2) - 2B \,(c_1 - c_2) \,X \vspace*{3pt} \\
                          & = & - \ (c_1 + c_2) + B \,(c_1 - c_2)(1 - 2X) \vspace*{3pt} \\
                          & \leq & - \ (c_1 + c_2) - B \,(c_1 - c_2) = - (B + 1) \,c_1 + (B - 1) \,c_2 < 0, \end{array} $$
 the limit~$\bar u_1 (B)$ is decreasing on~$(0, 1)$. \vspace*{5pt} \\
 Assume now that~$B \geq 1$. Since
 $$ \begin{array}{rcl}
     \bar u_1' (B) & = & \displaystyle \frac{1}{2B^2} \,\bigg(\frac{c_1 + c_2}{c_1 - c_2} + \sqrt{(B - 1)^2 + \frac{4 \,c_1 c_2}{(c_1 - c_2)^2}} \bigg) \vspace*{8pt} \\
                   &   & \displaystyle - \ \frac{B - 1}{2B} \bigg/ \sqrt{(B - 1)^2 + \frac{4 \,c_1 c_2}{(c_1 - c_2)^2}} \end{array} $$
 using again~\eqref{eq:decrease-1} gives
 $$ \bar u_1' (B) = - \ \frac{B - 1}{2B} \bigg/ \sqrt{(B - 1)^2 + \frac{4 \,c_1 c_2}{(c_1 - c_2)^2}} \ \leq \ 0 $$
 showing that~$\bar u_1 (B)$ is also decreasing on~$[1, 2)$.
\end{proof} \\ \\
 As explained in the introduction, the value in the limit as~$B \to 0$ and at~$B = 2$ can be understood returning to the stochastic model.
 We now give a proof using the expression~\eqref{eq:complete-limit-0}.
\begin{lemma} --
\label{lem:first-limit}
 We have~$\lim_{B \to 0} \,\bar u_1 (B) = \bar v_1$.
\end{lemma}
\begin{proof}
 Using a Taylor expansion and~\eqref{eq:decrease-1}, we get
 $$ \begin{array}{rcl}
    \displaystyle \sqrt{(B - 1)^2 + \frac{4 \,c_1 c_2}{(c_1 - c_2)^2}} & = & \displaystyle \sqrt{1 + \frac{4 \,c_1 c_2}{(c_1 - c_2)^2}} - B \bigg/ \sqrt{1 + \frac{4 \,c_1 c_2}{(c_1 - c_2)^2}} + o (B) \vspace*{8pt} \\
                                                                       & = & \displaystyle  - \ \frac{c_1 + c_2}{c_1 - c_2} + \frac{c_1 - c_2}{c_1 + c_2} \ B + o (B) \end{array} $$
 when~$B$ is small. In particular,
 $$ \begin{array}{rcl}
    \lim_{B \to 0} \,\bar u_1 (B) & = & 1/2 - \lim_{B \to 0} \ \displaystyle \frac{1}{2B} \,\bigg(\frac{c_1 + c_2}{c_1 - c_2} + \sqrt{(B - 1)^2 + \frac{4 \,c_1 c_2}{(c_1 - c_2)^2}} \bigg) \vspace*{8pt} \\
                                  & = & 1/2 - \lim_{B \to 0} \ \displaystyle \frac{1}{2B} \,\bigg(\frac{c_1 - c_2}{c_1 + c_2} \ B + o (B) \bigg) \vspace*{10pt} \\
                                  & = & \displaystyle \frac{1}{2} - \frac{1}{2} \ \frac{c_1 - c_2}{c_1 + c_2} = \frac{c_2}{c_1 + c_2} = \bar v_1. \end{array} $$
 This completes the proof.
\end{proof}
\begin{lemma} --
\label{lem:second-limit}
 We have~$\bar u_1 (1) = \sqrt{c_2} \ (\sqrt{c_1} + \sqrt{c_2})^{-1}$.
\end{lemma}
\begin{proof}
 Taking~$B = 1$ in the expression~\eqref{eq:complete-limit-0}, we get
 $$ \begin{array}{rcl}
    \bar u_1 (1) & = & \displaystyle \frac{1}{2} - \frac{1}{2} \bigg(\frac{c_1 + c_2}{c_1 - c_2} \bigg) - \frac{1}{2} \,\sqrt{\frac{4 \,c_1 c_2}{(c_1 - c_2)^2}} \vspace*{8pt} \\
                 & = & \displaystyle - \ \frac{c_2}{c_1 - c_2} - \sqrt{\frac{c_1 c_2}{(c_1 - c_2)^2}} = \frac{\sqrt{c_2} \,(\sqrt{c_2} - \sqrt{c_1})}{c_2 - c_1} = \frac{\sqrt{c_2}}{\sqrt{c_2} + \sqrt{c_1}}. \end{array} $$
 This completes the proof.
\end{proof}
\begin{lemma} --
\label{lem:third-limit}
 We have~$\bar u_1 (2) = 1/2$.
\end{lemma}
\begin{proof}
 Using~\eqref{eq:decrease-1} once more, we get
 $$ \begin{array}{rcl}
    \bar u_1 (2) & = & \displaystyle \frac{1}{2} - \frac{1}{4} \,\bigg(\frac{c_1 + c_2}{c_1 - c_2} + \sqrt{1 + \frac{4 \,c_1 c_2}{(c_1 - c_2)^2}} \bigg) \vspace*{8pt} \\
                 & = & \displaystyle \frac{1}{2} - \frac{1}{4} \,\bigg(\frac{c_1 + c_2}{c_1 - c_2} - \frac{c_1 + c_2}{c_1 - c_2} \bigg) = 1/2. \end{array} $$
 This completes the proof.
\end{proof} \\ \\
 The combination of Lemmas~\ref{lem:decrease}--\ref{lem:third-limit} gives the second part of the theorem.


\section{Proof of Theorem~\ref{th:bipartite} (bipartite graphs)}
\label{sec:bipartite}

\indent This section is devoted to the proof of Theorem~\ref{th:bipartite}.
 Throughout this section, we assume that the communication network is a finite, connected, bipartite graph.
 Recall that a graph is bipartite when there exists a partition~$\{V_1, V_2 \}$ of the vertex set such that
\begin{equation}
\label{eq:bipartite}
  (x, y) \in E \quad \hbox{implies that} \quad (x, y) \in V_1 \times V_2 \quad \hbox{or} \quad (x, y) \in V_2 \times V_1.
\end{equation}
 First, we prove that, when~$\ep = 0$, the two configurations
 $$ \xi_+ := \ind_{V_1} + 2 \times \ind_{V_2} \quad \hbox{and} \quad \xi_- := \ind_{V_2} + 2 \times \ind_{V_1} $$
 are the only two absorbing states, from which the theorem easily follows when~$\ep = 0$.
 Note that these configurations are simply the configurations in which all the vertices in the same bipartite set perform the same task and all the individuals in the other bipartite set perform the other task.
 To deal with positive defection probability, in which case the process becomes irreducible and converges weakly to a unique stationary distribution according to the proof of Lemma~\ref{lem:general}, the key ingredient is to prove that the
 fraction of time the process spends in one of the two configurations~$\xi_{\pm}$ in the long run can be made arbitrarily close to one by choosing~$\ep > 0$ small enough.
\begin{lemma} --
\label{lem:bipartite-absorbing}
 Let~$\ep = 0$.
 Then~$\xi_-$ and~$\xi_+$ are the only two absorbing states.
\end{lemma}
\begin{proof}
 Since~$\ep = 0$ and~$G$ is connected, \eqref{eq:rate-labor} becomes
 $$ \begin{array}{rcl}
      1 \ \to \ 2 & \hbox{at rate} & c_1 \,(1 - f_2 (x, \xi)) = c_1 \,f_1 (x, \xi) \vspace*{2pt} \\
      2 \ \to \ 1 & \hbox{at rate} & c_2 \,(1 - f_1 (x, \xi)) = c_2 \,f_2 (x, \xi). \end{array} $$
 It follows that~$\xi$ is an absorbing state if and only if
\begin{equation}
\label{eq:bipartite-absorbing-1}
  f_i (x, \xi) = 0 \quad \hbox{for all} \quad x \in V \quad \hbox{such that} \quad \xi (x) = i
\end{equation}
 for~$i = 1, 2$, indicating that no two neighbors perform the same task.
 In view of~\eqref{eq:bipartite}, the two configurations~$\xi_-$ and~$\xi_+$ clearly satisfy this property so they are absorbing states.
 To prove that there are no other absorbing states, fix~$\xi$ different from both~$\xi_{\pm}$.
 In such a configuration, there are two individuals in the same bipartite set that are performing different tasks.
 By obvious symmetry, we may assume without loss of generality that
\begin{equation}
\label{eq:bipartite-absorbing-2}
  \xi (x) \neq \xi (y) \quad \hbox{for some} \quad x, y \in V_1.
\end{equation}
 Now, since the graph is connected and~\eqref{eq:bipartite} holds, there exists a path with an even number of edges connecting vertex~$x$ and vertex~$y$, i.e., there exist~$x_0, x_1, \ldots, x_{2n}$ such that
\begin{equation}
\label{eq:bipartite-absorbing-3}
  x_0 = x \quad \hbox{and} \quad x_{2n} = y \quad \hbox{and} \quad (x_0, x_1), (x_1, x_2), \ldots, (x_{2n - 1}, x_{2n}) \in E.
\end{equation}
 It follows from~\eqref{eq:bipartite-absorbing-2}--\eqref{eq:bipartite-absorbing-3} that
 $$ \xi (x_j) = \xi (x_{j + 1}) \quad \hbox{for some} \quad j = 0, 1, \ldots, 2n - 1, $$
 showing that there are two neighbors performing the same task.
 In particular, \eqref{eq:bipartite-absorbing-1} fails, which implies that configuration~$\xi$ is not an absorbing state.
\end{proof} \\ \\
 It follows from Lemma~\ref{lem:bipartite-absorbing} and the fact that the graph is connected that, when~$\ep = 0$, the process has exactly three communication classes, namely
 $$ C_- := \{\xi_- \} \qquad C_0 := \{\xi \in \{1, 2 \}^V : \xi \notin \{\xi_-, \xi_+ \} \} \qquad C_+ := \{\xi_+ \}. $$
 The classes~$C_-$ and~$C_+$ are closed but~$C_0$ is not.
 Since in addition the graph is finite, the process gets trapped eventually in one of its two absorbing states. \\
\indent By the proof of Lemma~\ref{lem:general}, when~$\ep > 0$, there is a unique stationary distribution~$\pi$ to which the process converges weakly starting from any initial state.
 To establish the theorem in this case, the idea is to prove that the fraction of time spent in~$\xi_{\pm}$ can be made arbitrarily large by choosing the defection probability sufficiently small.
 Using the stationary distribution~$\pi$, this statement can be expressed as follows:
 for all~$\rho > 0$, there exists~$\ep_0 > 0$ such that
\begin{equation}
\label{eq:bipartite-2}
  P_{\pi} (\xi = \xi_+) + P_{\pi} (\xi = \xi_-) > 1 - \rho \quad \hbox{for all} \quad \ep < \ep_0.
\end{equation}
 To show~\eqref{eq:bipartite-2}, we define the stopping times
 $$ \begin{array}{rcl}
      T_{in} & := & \inf \,\{t : \xi_t \in \{\xi_-, \xi_+ \} \} = \inf \,\{t : \xi_t \notin C_0 \} \vspace*{4pt} \\
     T_{out} & := & \inf \,\{t : \xi_t \notin \{\xi_-, \xi_+ \} \} = \inf \,\{t : \xi_t \in C_0 \} \end{array} $$
 and give upper/lower bounds for the expected values
 $$ \begin{array}{rcl}
     \tau_{in} & := & \sup_{\ep, \xi_0} \,E \,(T_{in} \,| \,\xi_0 \notin \{\xi_-, \xi_+ \}) = \sup_{\ep, \xi_0} \,E \,(T_{in} \,| \,\xi_0 \in C_0) \vspace*{4pt} \\
    \tau_{out} & := & \inf_{\xi_0} \,E \,(T_{out} \,| \,\xi_0 \in \{\xi_-, \xi_+ \}) = \inf_{\xi_0} \,E \,(T_{out} \,| \,\xi_0 \notin C_0). \end{array} $$
 This is done in the next two lemmas.
\begin{lemma} --
\label{lem:bipartite-in}
 We have~$\tau_{in} < \infty$.
\end{lemma}
\begin{proof}
 Fix~$\xi \in C_0$.
 Then~\eqref{eq:bipartite-absorbing-1} does not hold so, when the system is in configuration~$\xi$, there exist two neighbors~$x, y$ performing the same task, say task~$i$.
 In particular, the rate at which vertex~$x$ switches task is bounded from below by
 $$ c_i \,(\ep + (1 - \ep) \,f_i (x, \xi)) \geq \,c_1 \,(\ep + (1/N)(1 - \ep)) \geq c_1 / N > 0. $$
 Since this lower bound does not depend on~$\ep$ and since the communication class~$C_0$ is finite and not closed, we deduce that
 $$ \begin{array}{l} c (\xi) := \sup_{\ep} \,E \,(T_{in} \,| \,\xi_0 = \xi) < \infty \quad \hbox{for all} \quad \xi \in C_0. \end{array} $$
 Using again that~$C_0$ is finite, we conclude that
 $$ \begin{array}{l} \tau_{in} = \sup_{\xi \in C_0} \,c (\xi) < \infty. \end{array} $$
 This completes the proof.
\end{proof}
\begin{lemma} --
\label{lem:bipartite-out}
  We have~$\tau_{out} \geq (\ep N c_2)^{-1} $.
\end{lemma}
\begin{proof}
 Intuitively, the result follows from the fact that the transition rates of the process are continuous with respect to~$\ep$ and the fact that, when~$\ep = 0$, the expected time~$\tau_{out}$ is infinite because the
 process cannot leave its absorbing states.
 To make this precise, let~$x$ be a vertex performing say task~$i$.
 Then, according to~\eqref{eq:bipartite-absorbing-1}, the rate at which this vertex switches its task given that the system is in one of the two configurations~$\xi_{\pm}$ is bounded by
 $$ c_i \,(\ep + (1 - \ep) \,f_i (x, \xi_{\pm})) = \ep \,c_i \leq \ep \,c_2. $$
 Since there are~$N$ vertices, standard properties of independent Poisson processes imply that the rate at which the process jumps in the set~$C_0$ is bounded by
 $$ \begin{array}{l} \lim_{h \to 0} \,(1/h) \,P \,(\xi_{t + h} \in C_0 \,| \,\xi_t \notin C_0) \leq \ep N c_2. \end{array} $$
 In conclusion, we have
 $$ \begin{array}{l} \tau_{out} = \inf_{\xi_0} \,E \,(T_{out} \,| \,\xi_0 \notin C_0) \geq E \,(\exponential (\ep N c_2)) = (\ep N c_2)^{-1} \end{array} $$
 which completes the proof.
\end{proof} \\ \\
 With Lemmas~\ref{lem:bipartite-absorbing}--\ref{lem:bipartite-out}, we are now ready to prove the theorem. \\ \\
\begin{demo}{Theorem~\ref{th:bipartite}} --
 The result is obvious for~$\ep = 0$.
 Indeed, in this case, Lemma~\ref{lem:bipartite-absorbing} implies that the process gets trapped in one of its two absorbing states therefore
 $$ \begin{array}{l}
    \lim_{s \to \infty} \,\phi (s) = \min (N_1/N, N_2/N) \,\ind \{\xi_t = \xi_+ \ \hbox{for some} \ t \} \vspace*{4pt} \\ \hspace*{100pt}
                                      + \ \max (N_1/N, N_2/N) \,\ind \{\xi_t = \xi_- \ \hbox{for some} \ t \}. \end{array} $$
 In particular, the limit belongs to
 $$ \begin{array}{l}
    \{\min (N_1/N, N_2/N), \max (N_1/N, N_2/N) \} \vspace*{4pt} \\
    \hspace*{40pt} \subset ((1 - \rho) \min (N_1/N, N_2/N), (1 + \rho) \max (N_1/N, N_2/N)). \end{array} $$
 Now, assume that~$\ep > 0$.
 Then, according to the proof of Lemma~\ref{lem:general}, the process converges weakly to a unique stationary distribution~$\pi$ so standard Markov chain theory implies that the fraction of time the process spends in one of
 the configurations~$\xi_{\pm}$ converges almost surely to
\begin{equation}
\label{eq:bipartite-fraction-1}
  \pi (\xi_-) + \pi (\xi_+) = P_{\pi} (\xi_t \notin C_0) \geq \tau_{out} \,(\tau_{in} + \tau_{out})^{-1}.
\end{equation}
 But according to Lemmas~\ref{lem:bipartite-in}--\ref{lem:bipartite-out}, there exists~$\ep_0 > 0$ such that
\begin{equation}
\label{eq:bipartite-fraction-2}
  \tau_{out} \,(\tau_{in} + \tau_{out})^{-1} \geq (1 + \ep N c_2 \,\tau_{in}) \geq 1 - \rho \quad \hbox{for all} \quad \ep < \ep_0.
\end{equation}
 Combining~\eqref{eq:bipartite-fraction-1}--\eqref{eq:bipartite-fraction-2} together with~\eqref{eq:general}, we get
 $$ \begin{array}{rcl}
    \lim_{s \to \infty} \,\phi (s) & \geq & \min (N_1/N, N_2/N) \,P_{\pi} (\xi_t \notin C_0) \vspace*{4pt} \\
                                   & \geq & (1 - \rho) \min (N_1/N, N_2/N). \end{array} $$
 Similarly, since~$\max (N_1/N, N_2/N) \geq 1/2$, we have
 $$ \begin{array}{rcl}
    \lim_{s \to \infty} \,\phi (s) & \leq & P_{\pi} (\xi_t \in C_0) + \max (N_1/N, N_2/N) \,P_{\pi} (\xi_t \notin C_0) \vspace*{4pt} \\
                                   & \leq & \rho + (1 - \rho) \max (N_1/N, N_2/N) \vspace*{4pt} \\
                                   & \leq & (1 + \rho) \max (N_1/N, N_2/N). \end{array} $$
 This completes the proof of the theorem.
\end{demo}


\section{Proof of Theorem~\ref{th:lattice} (one-dimensional lattice)}
\label{sec:lattice}

\indent The final communication network we study is the one-dimensional integer lattice.
 Throughout this section, we assume that the initial configuration of the process is a Bernoulli product measure.
 The main objective is to prove that the limiting probability
\begin{equation}
\label{eq:limit-lattice}
  \begin{array}{l} \lim_{s \to \infty} \,P \,(\xi_s (x) = \xi_s (x + 1)) \quad \hbox{where} \quad x \in \Z \end{array}
\end{equation}
 can be made arbitrarily close to zero by choosing~$\ep > 0$ sufficiently small.
 Note that, since both the initial distribution and the evolution rules are translation invariant, this limiting probability does not depend on the particular choice of vertex~$x$.
 The theorem will follow from our estimate of the limit in~\eqref{eq:limit-lattice} and the ergodic theorem.

\indent The process on the one-dimensional lattice, or any (infinite) graph whose degree is uniformly bounded, can be constructed using the following rules and collections of independent Poisson processes,
 called a graphical representation~\cite{harris_1972}:
 for each oriented edge~$(x, y) \in E$,
\begin{itemize}
 \item we draw a solid arrow~$y \longrightarrow x$ at the times of a Poisson process with rate~$(1 - \ep) \,c_1$ to indicate that vertex~$x$ anti-imitates vertex~$y$, \vspace*{3pt}
 \item we draw a dashed arrow~$y \dasharrow x$ at the times of a Poisson process with rate~$(1 - \ep)(c_2 - c_1)$ to indicate that if vertex~$x$ performs task~2 then it anti-imitates vertex~$y$. \vspace*{3pt}
 \item we put a dot~$\bullet$ at vertex~$x$ at the times of a Poisson process with rate~$\ep \,c_1$ to indicate that vertex~$x$ switches task, \vspace*{3pt}
 \item we put a cross~$\times$ at vertex~$x$ at the times of a Poisson process with rate~$\ep \,(c_2 - c_1)$ to indicate that if vertex~$x$ performs task~2 then it switches task.
\end{itemize}
 Since our objective is to compare the tasks performed by neighbors, rather than studying the dynamics of tasks on the set of vertices, it is more convenient to keep track of the dynamics of agreements along the edges.
 To do this, we put a type~$i$ particle on edge~$e$ if and only if the two individuals connected by this edge perform the same task~$i$.
 Returning to the special case of the one-dimensional lattice, this results in a process~$\zeta_t$ defined as
 $$ \zeta_t ((x, y)) := \ind \{\xi (x) = \xi (y) = 1 \} + 2 \times \ind \{\xi (x) = \xi (y) = 2 \} $$
 for all~$x = y \pm 1 \in \Z$.
 In particular, edges in state zero, that we call empty edges, connect individuals performing different tasks.
 The process~$\zeta_t$ is not Markov because the configuration in which all the edges are empty can result in two different configurations of tasks on the vertices.
 It can be proved, however, that the pair~$(\zeta_t, \xi_t (0))$ is a Markov process, but since this is not relevant to establish our theorem, we do not show this result.
 To understand the dynamics on the edges, we now look at the effect of the graphical representation on the configuration of particles.
\begin{itemize}
 \item Solid arrows~$y \longrightarrow x$ only affect the task at~$x$, and so the state of the two edges incident to~$x$, if and only if both vertices~$x$ and~$y$ perform the same task before the interaction.
  This leads to the two cases illustrated on the left-hand side of Figure~\ref{fig:coupling} as well as the additional two cases obtained by exchanging the roles of task~1 and task~2. \vspace*{3pt}
 \item Dashed arrows have the same effect as solid arrows but only if both vertices~$x$ and~$y$ perform task~2 before the interaction. \vspace*{3pt}
 \item Similarly, dots at vertex~$x$ only affect the task at~$x$, and so only the state of the two edges incident to vertex~$x$.
  This leads to the four cases illustrated on the right-hand side of Figure~\ref{fig:coupling} as well as the additional four cases obtained by exchanging the roles of the tasks. \vspace*{3pt}
 \item Crosses have the same effect as dots only if~$x$ performs task~2 before the interaction.
\end{itemize}
\begin{figure}[t]
\centering
\includegraphics[width=0.75\textwidth]{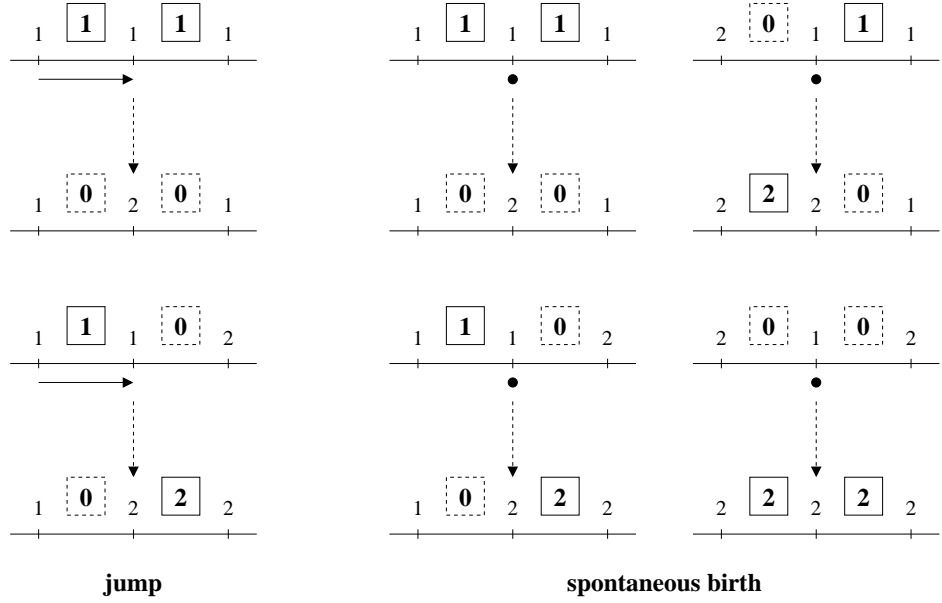}
\caption{\upshape{Coupling with the system of type~1 and type~2 particles.}}
\label{fig:coupling}
\end{figure}
 In view of the transitions in Figure~\ref{fig:coupling} and the rates of the Poisson processes used in the graphical representation, the dynamics on the edges can be summarized as follows.
\begin{itemize}
 \item Type~$i$ particles jump at rate~$(1 - \ep) \,c_i$ to the right or to the left with probability one half and change their type at each jump. \vspace*{3pt}
 \item There are spontaneous births of pairs of particles of type~$i$ at rate~$\ep \,c_i$ at edges incident to a vertex not performing task~$i$. \vspace*{3pt}
 \item When two particles occupy the same edge as the result of a jump or a spontaneous birth, both particles annihilate.
\end{itemize}
 Using the dynamics of type~1 and type~2 particles, we now study the probability in~\eqref{eq:limit-lattice}, which is the probability that the edge connecting vertices~$x$ and~$x + 1$ is empty.

\indent To prove that the limiting probability in~\eqref{eq:limit-lattice} can be made arbitrarily small, we use the construction shown in Figure~\ref{fig:particles}.
 The proof also requires three simple estimates for the probability of the three events illustrated in the picture.
 These three estimates are given in the following three lemmas, where~$\rho$ is a small positive constant like in the statement of the theorem.
 First, we look at the special case when the defection probability is equal to zero.
\begin{lemma} --
\label{lem:lattice-first}
 Let $\ep = 0$. Then, there exists~$T_1 < \infty$ such that
 $$ \begin{array}{l} \sup_{\xi} \,P \,(\xi_t (x) = \xi_t (x + 1) \,| \,\xi_{t - T} = \xi) < \rho / 3 \quad \hbox{for all} \quad t \geq T \geq T_1. \end{array} $$
\end{lemma}
\begin{proof}
 Note that, when~$\ep = 0$, there is no spontaneous birth so the system of particles reduces to a system of annihilating symmetric random walks on the one-dimensional lattice that jump either at rate~$c_1$ or at rate~$c_2$.
 Since such a system goes extinct, there is~$T_1 < \infty$ such that
 $$ \begin{array}{l}
    \sup_{\xi} \,P \,(\xi_T (x) = \xi_T (x + 1) \,| \,\xi_0 = \xi) \vspace*{4pt} \\ \hspace*{25pt} = \
    \sup_{\zeta} \,P \,(\zeta_T ((x, x + 1)) \neq 0 \,| \,\zeta_0 = \zeta) < \rho / 3 \quad \hbox{for all} \quad T \geq T_1, \end{array} $$
 but since~$\xi_t$ is a time-homogeneous Markov chain,
 $$ \begin{array}{l}
    \sup_{\xi} \,P \,(\xi_t (x) = \xi_t (x + 1) \,| \,\xi_{t - T} = \xi) \vspace*{4pt} \\ \hspace*{25pt} = \
    \sup_{\xi} \,P \,(\xi_T (x) = \xi_T (x + 1) \,| \,\xi_0 = \xi) < \rho / 3 \quad \hbox{for all} \quad t \geq T \geq T_1. \end{array} $$
 This completes the proof.
\end{proof} \\ \\
 The second step of the proof is to show that, with probability close to one, the set of space-time points that may have influenced the state of edge~$(x, x + 1)$ at time~$t$, which is traditionally called the influence set in the field of
 interacting particle systems, cannot grow too fast.
 Referring to the graphical representation, the influence set is the set of space-time points from which there is an oriented path of solid and dashed arrows leading to either~$x$ or~$x + 1$ at time~$t$.
 More formally, the influence set is defined as
\begin{equation}
\label{eq:influence-set}
  I_t (x, x + 1) := \{(y, s) \in \Z \times \R_+ : (y, s) \leadsto (x, t) \ \hbox{or} \ (y, s) \leadsto (x + 1, t) \}
\end{equation}
 where~$(y, s) \leadsto (z, t)$ means that there exist
 $$ y = x_0, x_1, \ldots, x_n = z \in \Z \quad \hbox{and} \quad s_0 < s_1 < \cdots < s_{n + 1} $$
 such that, for~$j < n$, there is an arrow~$x_j \longrightarrow x_{j + 1}$ or~$x_j \dasharrow x_{j + 1}$ at time~$s_{j + 1}$.
 To state our next two lemmas, we also introduce
 $$ \begin{array}{rcl}
      J_T & := & \{(y, s) \in \Z \times \R_+ : - \ceil{c_2 T} < y < \ceil{c_2 T} \ \hbox{and} \ s \leq T \} \vspace*{4pt} \\
      K_T & := & \{(y, s) \in \Z \times \R_+ : - \ceil{c_2 T} \leq y \leq \ceil{c_2 T} \ \hbox{and} \ s \leq T \} \end{array} $$
 where~$\ceil{r}$ is the smallest integer not less than~$r$.
\begin{lemma} --
\label{lem:lattice-second}
 There exists~$T_2 < \infty$ such that, for all~$t \geq T \geq T_2$,
 $$ P \,(I_t (x, x + 1) \cap (\Z \times [t - T, t]) \not \subset (x, t - T) + J_T) < \rho / 3. $$
\end{lemma}
\begin{proof}
 For all~$0 \leq s \leq t$, we write
 $$ I_t (x, x + 1) \cap (\Z \times \{t - s \}) = ([l_s, r_s] \cap \Z) \times \{t - s \}. $$
 In other words, $l_s$ and~$r_s$ are the spatial location of the leftmost point and the spatial location of the rightmost point in the influence set at time~$t - s$.
 Since arrows of either type from a given vertex to a given neighbor appear in the graphical representation at rate at most~$c_2 / 2$, it follows from the definition of the influence set that
 $$ l_s = x - \poisson (c_2 \,s / 2) \quad \hbox{and} \quad r_s = (x + 1) + \poisson (c_2 \,s / 2) $$
 in distribution for all~$0 \leq s \leq t$.
 In particular, standard large deviation estimates for the Poisson random variable imply that there exists~$T_2 < \infty$ such that
 $$ \begin{array}{l}
      P \,(I_t (x, x + 1) \cap (\Z \times [t - T, t]) \not \subset (x, t - T) + J_T) \vspace*{4pt} \\ \hspace*{25pt} = \
      P \,(I_t (x, x + 1) \cap (\Z \times \{t - T \}) \not \subset (x - \ceil{c_2 T}, x + \ceil{c_2 T}) \times \{t - T \}) \vspace*{4pt} \\ \hspace*{25pt} \leq \
      P \,(l_T \leq x - \ceil{c_2 T}) + P \,(r_T \geq x + \ceil{c_2 T}) \vspace*{4pt} \\ \hspace*{25pt} \leq \
      2 \,P \,(\poisson (c_2 T / 2) \geq \ceil{c_2 T} - 1) \leq \exp (- c_2 T / 8) < \rho / 3
    \end{array} $$
 for all~$t \geq T \geq T_2$.
 This completes the proof.
\end{proof} \\ \\
 The third step of the proof is to show that, with probability close to one when the defection probability is small, there are no spontaneous births in a translation of~$K_T$.
\begin{lemma} --
\label{lem:lattice-third}
 For all~$T$, there exists~$\ep_0 > 0$ such that, for all~$t \geq T$,
 $$ P \,(\hbox{there is a $\bullet$ or a $\times$ in $(x, t - T) + K_T$}) < \rho / 3 \quad \hbox{for all} \quad \ep < \ep_0. $$
\end{lemma}
\begin{proof}
 Dots and crosses appear at each vertex altogether at rate~$\ep c_2$ hence
 $$ \hbox{number of~$\bullet$'s and~$\times$'s in~$(x, t - T) + K_T$} = \poisson ((2 \,\ceil{c_2 T} + 1) \,T \ep c_2) $$
 in distribution.
 This implies that, for all~$T$,
 $$ \begin{array}{l}
      P \,(\hbox{there is a $\bullet$ or a $\times$ in $(x, t - T) + K_T$}) \vspace*{4pt} \\ \hspace*{50pt} = \
      P \,(\poisson ((2 \,\ceil{c_2 T} + 1) \,T \ep c_2) \neq 0) \vspace*{4pt} \\ \hspace*{50pt} \leq \
      1 - \exp (- (2 c_2 T + 3) \,T \ep c_2) \leq (2 c_2 T + 3) \,T \ep c_2 \end{array} $$
 can be made smaller than~$\rho / 3$ by choosing~$\ep > 0$ small.
\end{proof} \\ \\
 Lemmas~\ref{lem:lattice-first}--\ref{lem:lattice-third} identify three events with small probability.
 The complement of these events, which occur with high probability, are illustrated schematically in Figure~\ref{fig:particles}.
 With these three lemmas in hands, we are now ready to prove the theorem. \\ \\
\begin{figure}[t]
 \centering
 \scalebox{0.50}{\input{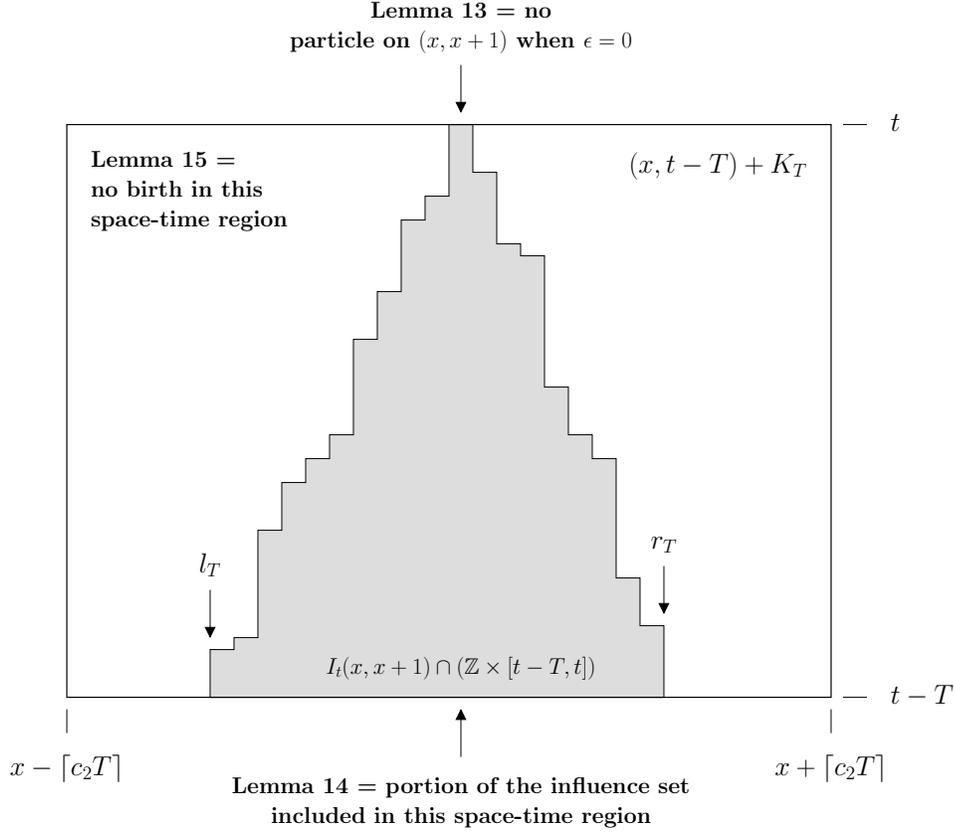}}
 \caption{\upshape{Picture of the proof of Theorem~\ref{th:lattice}.}}
\label{fig:particles}
\end{figure}
\begin{demo}{Theorem~\ref{th:lattice}} --
 Let~$\rho > 0$ small, then
\begin{itemize}
 \item fix~$T_1$ and~$T_2$ such that the conclusions of Lemmas~\ref{lem:lattice-first}--\ref{lem:lattice-second} both hold and \vspace*{3pt}
 \item fix~$\ep_0 > 0$ such that the conclusion of Lemma~\ref{lem:lattice-third} holds for~$T := \max (T_1, T_2)$.
\end{itemize}
 For these parameters, we will prove that
\begin{equation}
\label{eq:lattice-1}
  P \,(\xi_t (x) = \xi_t (x + 1)) < \rho \quad \hbox{for all} \quad (x, t) \in \Z \times (T, \infty) \ \hbox{and} \ \ep < \ep_0.
\end{equation}
 To do this, let~$\xi_{\cdot}^0$ be the process starting from the same initial configuration and constructed from the same graphical representation as~$\xi_{\cdot}$ but ignoring the dots and crosses in the graphical representation
 after time~$t - T$, where time~$t > T$ is fixed.
 In particular,
\begin{equation}
\label{eq:lattice-2}
  \xi_s (y) = \xi_s^0 (y) \quad \hbox{for all} \quad (y, s) \in \Z \times [0, t - T].
\end{equation}
 Note also that the influence sets~\eqref{eq:influence-set} of both processes coincide since both processes are constructed from the same collections of arrows and since the dots and the crosses are irrelevant in the construction
 of the influence sets.
 This, together with~\eqref{eq:lattice-2}, implies that
\begin{equation}
\label{eq:lattice-3}
  \xi_t (x) = \xi_t^0 (x) \quad \hbox{and} \quad \xi_t (x + 1) = \xi_t^0 (x + 1)
\end{equation}
 unless there is a dot or a cross after time~$t - T$, when the two processes may disagree, in the neighborhood of the influence set, i.e.,
 $$ \hbox{there is a~$\bullet$ or a~$\times$ in} \ (I_t (x, x + 1) + [-1, 1]) \cap (\Z \times [t - T, t]), $$
 which occurs whenever
\begin{equation}
\label{eq:lattice-4}
  \begin{array}{l}
    I_t (x, x + 1) \cap (\Z \times [t - T, t]) \not \subset (x, t - T) + J_T \vspace*{3pt} \\ \hspace*{40pt}
  \hbox{or \ there is a $\bullet$ or a $\times$ in the region $(x, t - T) + K_T$}. \end{array}
\end{equation}
 Combining~\eqref{eq:lattice-3}--\eqref{eq:lattice-4} and applying Lemmas~\ref{lem:lattice-first}--\ref{lem:lattice-third}, we deduce that
 $$ \begin{array}{l}
      P \,(\xi_t (x) = \xi_t (x + 1)) \leq P \,(\xi_t^0 (x) = \xi_t^0 (x + 1)) + P \,(\xi_t \neq \xi_t^0 \ \hbox{on} \ \{x, x + 1 \}) \vspace*{4pt} \\ \hspace*{50pt}
                                      \leq P \,(\xi_t^0 (x) = \xi_t^0 (x + 1)) + P \,(\hbox{there is a $\bullet$ or a $\times$ in $(x, t - T) + K_T$}) \vspace*{4pt} \\ \hspace*{80pt} + \
                                           P \,(I_t (x, x + 1) \cap (\Z \times [t - T, t]) \not \subset (x, t - T) + J_T) \vspace*{4pt} \\ \hspace*{50pt}
                                        < \rho / 3 + \rho / 3 + \rho / 3 = \rho. \end{array} $$
 This shows~\eqref{eq:lattice-1}. In particular, for all~$t > T$,
\begin{equation}
\label{eq:lattice-5}
  \begin{array}{rcl}
     P \,(\xi_t (x) = 1) & = & P \,(\xi_t (x) = 1 \,| \,\xi_t (x) = \xi_t (x + 1)) \,P \,(\xi_t (x) = \xi_t (x + 1)) \vspace*{4pt} \\ && \hspace*{10pt} + \
                               P \,(\xi_t (x) = 1 \,| \,\xi_t (x) \neq \xi_t (x + 1)) \,P \,(\xi_t (x) \neq \xi_t (x + 1)) \vspace*{4pt} \\
                         & = & P \,(\xi_t (x) = 1 \,| \,\xi_t (x) = \xi_t (x + 1)) \,P \,(\xi_t (x) = \xi_t (x + 1)) \vspace*{4pt} \\ && \hspace*{10pt} + \
                                   (1/2) \,P \,(\xi_t (x) \neq \xi_t (x + 1)) \vspace*{4pt} \\
                         & \leq & 1 \times \rho + (1/2)(1 - \rho) = (1/2)(1 + \rho). \end{array}
\end{equation}
 Similarly, for all~$t > T$, we have
\begin{equation}
\label{eq:lattice-6}
  P \,(\xi_t (x) = 1) \geq 0 \times \rho + (1/2)(1 - \rho) = (1/2)(1 - \rho).
\end{equation}
 In addition, since the initial distribution of the process is a Bernoulli product measure and the Poisson processes in the graphical representation are independent, the ergodic theorem is applicable, from which it follows that, for all~$t > T$,
\begin{equation}
\label{eq:lattice-7}
  \begin{array}{l} \lim_{N \to \infty} \ \displaystyle \frac{1}{2N + 1} \,\sum_{|x| \leq N} \ind \{\xi_t (x) = 1 \} = P \,(\xi_t (x) = 1). \end{array}
\end{equation}
 Combining~\eqref{eq:lattice-5}--\eqref{eq:lattice-7}, we conclude that
 $$ \begin{array}{rcl}
      (1/2)(1 - \rho) & \leq & \lim_{s, N \to \infty} \,\phi_N (s) \vspace*{10pt} \\
                         & = & \lim_{s, N \to \infty} \,\displaystyle \frac{1}{s \,(2N + 1)} \,\int_0^s \sum_{|x| \leq N} \ind \{\xi_t (x) = 1 \} \,dt \vspace*{8pt} \\
                         & = & \lim_{s, N \to \infty} \,\displaystyle \frac{1}{(s - T)(2N + 1)} \,\int_T^s \sum_{|x| \leq N} \ind \{\xi_t (x) = 1 \} \,dt \vspace*{4pt} \\
                      & \leq & (1/2)(1 + \rho), \end{array} $$
 which completes the proof of the theorem.
\end{demo} \\


\noindent\textbf{Acknowledgment}.
 The authors would like to thank Yun Kang and Jennifer Fewell for the fruitful discussions about social insects that motivated this work.



\begin{thebibliography}{100}

\bibitem{beshers_fewell_2001}
 Beshers, S. N. and Fewell, J. H. (2001). Models of division of labor in social insects. \emph{Annual Review of Entomology} \textbf{46} 413--440.

\bibitem{clifford_sudbury_1973}
 Clifford, P. and Sudbury, A. (1973). A model for spatial conflict. \emph{Biometrika} \textbf{60} 581--588.

\bibitem{durrett_1995}
 Durrett, R. (1995). Ten lectures on particle systems. In \emph{Lectures on probability theory (Saint-Flour, 1993)}, volume 1608 of \emph{Lecture Notes in Math.}, pages 97--201. Springer, Berlin.


\bibitem{gordon_1992}
 Gordon, D. M., Goodwin, B. C. and Trainor, L. E. H. (1992). A parallel distributed model of the behavior of ant colonies. \emph{J. Theor. Biol.} \textbf{156} 293--307.

\bibitem{harris_1972}
 Harris, T. E. (1972). Nearest neighbor Markov interaction processes on multidimensional lattices. \emph{Adv. Math.} \textbf{9} 66--89.

\bibitem{holley_liggett_1975}
 Holley, R. A. and Liggett, T. M. (1975). Ergodic theorems for weakly interacting systems and the voter model. \emph{Ann. Probab.} \textbf{3} 643--663.


\bibitem{liggett_1985}
 Liggett, T. M. (1985). \emph{Interacting particle systems}, volume 276 of \emph{Grundlehren der Mathematischen Wissenschaften [Fundamental Principles of Mathematical Sciences]}. Springer-Verlag, New York.

\bibitem{liggett_1999}
 Liggett, T. M. (1999). \emph{Stochastic interacting systems: contact, voter and exclusion processes}, volume 324 of \emph{Grundlehren der Mathematischen Wissenschaften [Fundamental Principles of Mathematical Sciences]}. Springer-Verlag, Berlin.

\bibitem{matloff_1977}
 Matloff, N. S. (1977). Ergodicity conditions for a dissonant voting model. \emph{Ann. Probab.} \textbf{5} 371--386.

\bibitem{pacala_1974}
 Pacala, S. W., Gordon, D. M. and Godfray, H. C. J. (1996). Effects of social group size on information transfer and task allocation. \emph{Evol. Ecol.} \textbf{10} 127--165.

\bibitem{wilson_1971}
 Wilson, E. O. (1971). \emph{The Insect Societies}. Cambridge. MA: Harvard Univ. Press.

\bibitem{wilson_1980}
 Wilson, E. O. (1980). Caste and division of labor in leaf-cutter ants (Hymenoptera: Formicidae: \emph{Atta}). \emph{Behavioral Ecology and Sociobiology} \textbf{7} 157--165.

\end{thebibliography}
\end{document}